\documentclass{article}
\usepackage{amsmath}
\usepackage{amsthm}
\usepackage{amsfonts}
\usepackage{url}

\newtheorem{theorem}{Theorem}
\newtheorem{conjecture}{Conjecture}
\newtheorem{corollary}[theorem]{Corollary}
\newtheorem{definition}[theorem]{Definition}

\newtheorem{hypothesis}{Hypothesis}

\newtheorem{proposition}[theorem]{Proposition}

\newtheorem{axiom}{Axiom}

\newcommand\point{\mathrm{pt}}

\makeatletter
\@twosidetrue
\makeatother

\begin{document}
\title{Axioms for the $g$-vector of general convex polytopes}
\author{Jonathan Fine}
\date{18 November 2010}

\maketitle

\begin{abstract}
McMullen's $g$-vector is important for simple convex polytopes.  This
paper postulates axioms for its extension to general convex polytopes.
It also conjectures that, for each dimension $d$, a stated finite
calculation gives the formula for the extended $g$-vector.  This
calculation is done by computer for $d=5$ and the results analysed.
The conjectures imply new linear inequalities on convex polytope flag
vectors.  Underlying the axioms is a hypothesised higher-order homology
extension to middle perversity intersection homology (order-zero
homology), which measures the failure of lower-order homology to have a
ring structure.
\end{abstract}

\bigskip

\begin{quotation}
\noindent
How often have I said to you that when you have eliminated the
impossible, whatever remains, however improbable, must be the truth?

\ \hfill Sherlock Holmes in \emph{The Sign of the Four} by
Sir Arthur Conan Doyle
\end{quotation}


\section{Introduction}

This section describes the purpose and key concepts of this paper.
Every simple convex polytope $X$ has a $g$-vector $g(X)$, a linear
function of the face vector $f(X)$.  McMullen \cite{mcm71:num}
introduced the $g$-vector in 1971 to express his conjectured
conditions on $f(X)$. In 1980 Stanley~\cite{sta80:num} proved the
necessity of these conditions, and Billera and Lee~\cite{billee80:suf}
the sufficiency.  The distant goal is to find and prove the extension
of these conditions to general convex polytopes.

McMullen's conditions have three parts: (1)~linear equations on
$f(X)$, (2)~linear inequalities $g_i(X) \geq 0$, and (3)~pseudo-power
growth limits $g_{i+1}\leq g_i^{\langle i \rangle}$.  They correspond
to there being homology groups $H(X)$ that respectively satisfy
Poincar\'e duality, satisfy hard Lefschetz and are a ring generated by
the facets.  McMullen at the time was not aware of this connection.

The main results in this paper are: (1)~axioms which are conjectured
to determine the extended $g$-vector; (2)~a conjecture that reduces
this determination to a finite calculation; and (3)~this calculation
performed by computer in dimension $d=5$.  Almost all of the
calculation is using the axioms to eliminate the impossible.

Two partial extensions are already known. Bernstein, Khovanski and
MacPherson independently (see \cite[\S4.1]{dcml09:decomp}) found one
in around 1982. From a rational convex polytope $X$ a projective
algebraic variety $\mathbb{P}_X$ can be constructed (via toric
geometry).  The middle perversity intersection homology (mpih) Betti
numbers $h(X)$ of $\mathbb{P}_X$ are by Poincar\'e duality palindromic.
As in the simple case, $g_i(X) = h_i(X) - h_{i-1}(X)$ for $0 \leq 2i
\leq \dim X $ defines $g_i$.  This extends the simple definition.  By
hard Lefschetz, $g_i(X) \geq 0$.

This $g$-vector is a linear function not of the face vector of $X$
(which counts the number of faces on $X$ of each dimension) but of the
flag vector (which counts chains of inclusions between faces).  In
1985 Bayer and Billera~\cite{baybil85:gen} found another partial
extension, which determines $g$ up to an invertible matrix.  Their
result is:

\begin{theorem}[Generalised Dehn-Sommerville]
The vector space $\mathcal{F}_d$ spanned by the flag vectors of all
$d$-dimensional convex polytopes has dimension the Fibonacci number
$F_{d+1}$.  The flag vectors $f(v(\point))$ are a basis for
$\mathcal{F}_d$, where $v$ ranges over the $F_{d+1}$ degree~$d$ words
in $C$ and $IC$, as defined below.
\end{theorem}

\begin{definition}[$C$ and $I$]
Suppose $X$ is a convex polytope.  The \emph{cone} (or pyramid) $CX$
on $X$ is the convex hull of $X$ with a point, the \emph{apex}, not
lying in the affine linear span of $X$.  The \emph{cylinder} (or
prism) $IX$ on $X$ is the Cartesian product of an interval, say $[0,
  1]$, with $X$.
\end{definition}

Thus in dimension~$5$ the flag vector has $F_6 = 8$ independent
components while the $g$-vector has $\lfloor 5/2 \rfloor + 3 = 3$.  We
wish to extend $g$ so that it encode the whole of $f$.  Here are the
basic hypotheses made regarding the components $g_i$ of the extended
$g$-vector, besides being a linear function of the flag vector.  But
first we need a definition.

\begin{definition}[The $D$ operator]
The difference $IC-CC$ is the $D$ operator on convex polytopes.
\end{definition}

\begin{hypothesis}[Effective]
For all convex polytopes $X$ we have $g_i(X) \geq 0$.
\end{hypothesis}

\begin{hypothesis}[Zero-one on $(C,D)$ words]
$g_i(w(\point)) \in \{0, 1\}$ for all words $w$ in $C$ and $D$.
\end{hypothesis}

By generalised Dehn-Sommerville the flag vectors $f(w(\point))$ are a
basis for $\mathcal{F}_d$.  The second hypothesis now implies that
$g_i$ is determined by $s=\{w | g_i(w(\point)) = 1\}$.  For each $d$
there are finitely many (in fact $2^n$ where $n=F_{d+1}$)
possibilities for $s$ and thus $g_i$.  But if $g_i(X)< 0$ then $g_i$
fails the first hypothesis and so must be eliminated.  The next
hypothesis makes this elimination an explicit finite, but perhaps
large, calculation.

\begin{hypothesis}
The set $P_{01,d}$ of all $d$-dimensional $01$-polytopes and their
polars eliminate all impossible $g_i$.
\end{hypothesis}

Obtaining the $g$-vector from what whatever remains requires
additional steps that depend on precise concepts. The rest of this
paper is organised as follows.  After providing background and
notation we define and motivate the $(C,D)$ basis and the zero-one on
$(C,D)$ hypothesis.  In particular we discuss the already known
intersection homology of toric varieties, and a possible extension
that encodes the remainder of the flag vector.  Next we introduce
axioms to give a precise statement of our assumptions and conjectures,
and report on our calculations.  Finally we summarise and discuss some
open questions.

\section{Background}

This section establishes notation, provides some definitions and
provides other background.  Some important matters not relevant to the
core of this paper, such as the decomposition theorem, are left
unexplained.  Concepts used in only one section are generally not
placed here.

Throughout \emph{polytope} will mean convex polytope, $X$ will be a
polytope of dimension~$d$, $w$ a word of degree $d$ in $C$ and $D$ (as
defined in the Introduction) and $s$ will be a set of words~$w$ (also
known as a \emph{word-set}). Usually, $v$ is a word in $(C, IC)$ or in
$(C, I)$.

We will use $f(X)$ to denote the \emph{flag vector} of $X$ and $g(X)$
is our goal, the $g$-vector of $X$.  It has \emph{components} $g_i$.
Suppose $Y$ is a weighted formal sum $\sum \alpha_i X_i$ of polytopes.
By $f(Y)$ we will mean $\sum \alpha_if(X_i)$, and similarly for other
vector- and number-valued functions, such as $g$ and $g_i$.

We use $\lambda_w = \lambda_w(X)$ to denote coefficients in the
\emph{$(C, D)$-basis}.  (This should not be confused with the
$cd$-index of convex polytopes.)  We use $\lambda_s$ for $\sum_{w\in
  s}\lambda_w$.  We use $\deg_C w$ and $\deg_D w$ to denote the number
of $C$'s and $D$'s in $w$, and write $\deg w = \deg_C w + 2 \deg_D w$.
The \emph{order} of a word is the number of occurences of $CD$ in it.

Every polytope $X$ has a \emph{polar} $X^\vee$.  This reverses
inclusions on faces and so $f(X)$ and $f(X^\vee)$ are linear functions
of each other.  A polytope is \emph{simple} if its facets
(codimension~1 faces) are in general position, and \emph{simplicial}
if its vertices are in general position.  Simple and simplicial
correspond under polarisation.

Combinatorists usually start with simplicial polytopes but for
homology simple polytopes are a better starting point.  Here is an
example. The polar of the $I$ operator is the bipyramid $B$, and
\cite{baybil85:gen} used $B$ not $I$.  The $IC$ equation reduces $IC -
CC$ to a Cartesian product operator.  The polar of a Cartesian product
not easy to describe.  This paper takes the simple/homology point of
view.  For us, McMullen's conditions apply to simple polytopes, even
though he formulated them for simplicial polytopes.

From every rational polytope $X$ a toric variety $\mathbb{P}_X$ can be
constructed.  It has homology groups, which we denote by $H(X)$.  The
definition of $H(X)$ extends to all polytopes, rational or not. We
will index homology groups by complex dimension (as the others are
zero).  Homology will always mean middle perversity intersection
homology, which for simple polytopes is the same as ordinary homology.
In \S\ref{sect:higher} we hypothesise the existence of
\emph{higher-order} homology groups, and call the already known
part of $H(X)$ the \emph{order-zero} homology.

We use $[a,b,c]$ to denote the polynomial $a + bx + cx^2$, and think
of the order-zero part of $h(X)$ and $g(X)$ as polynomials in $x$.
For example, $h(X) = \sum x^i h_i(X)$.  This allows them to be
multiplied.  Thus, $h(\mathbb{P}_2)= [1, 1, 1]$ and $h(\mathbb{P}_2
\times \mathbb{P}_3) = [1, 1, 1][1, 1, 1, 1]$ by the K\"unneth
formula.

The \emph{hard Lefschetz theorem} says that the hyperplane class
$\omega\in H_{d-1}(X)$ induces an isomorphism $\omega^{j-i}:H_i(X) \to
H_j(X)$ when $i\geq j$ and $i+j=d$.  A class $\eta$ is
\emph{primitive} if $\omega^{j-1+1}(\eta) = 0$.

If $\mathcal{B}$ is a Boolean expression we let $(\mathcal{B})$ denote
$1$ if $\mathcal{B}$ is true and $0$ otherwise.  Thus, $\delta_{ij} =
(i = j)$ is the Kronecker delta and $(w \in s)$ is the characteristic
function of a set $s$.

Finally, some miscellaneous definitions and notations.  A $d$-cube, or
\emph{cube} for short, will mean a $d$-dimensional hypercube.  We do
not need in this paper the definition of the \emph{pseudo-power}
$n^{\langle i \rangle}$.  The interested reader can find it in
\cite{mcm71:num} or \cite{bayerlee94:aspects}.  We use $\lfloor x
\rfloor$ to denote the largest integer not greater than $x$.
Hypotheses, axioms and conjectures are numbered separately.
Everything else is numbered together.


\section{The $(C, D)$ basis}

According to generalised Dehn-Sommerville every polytope flag vector
has a unique expression as a linear combination of the flag vectors of
$v(\point)$ where $v$ is a word in $C$ and $IC$.  This applies in
particular to $v(\point)$ where $v$ is any word in $C$ and $I$ that is
\emph{not} a word in $C$ and $IC$.  This suggests that there is a
relationship between $I$ and $C$.  In fact~\cite{fine95:mvic}:

\begin{theorem}[The $IC$ equation]
As operators on flag vectors, $I$ and $C$ satisfy the equation
\begin{equation*}
  (IC - CC) \> I \> \equiv \> I \> (IC - CC)\>.
\end{equation*}
\end{theorem}

\begin{proof}
Because it is central to the properties of $D$, we summarise the proof
given in~\cite{fine95:mvic}.  Let $X$ be a convex polytope.  By
definition $CCX$ has two apexes, but there is no geometric difference
between the first and the second.  Put another way, $CCX$ is the join
of $X$ with an interval and that interval is the \emph{apex edge} of
$CCX$.  Similarly, $ICX$ has an apex edge.  Along their apex edges
$CCX$ and $ICX$ have the same combinatorial structure and so
\begin{equation*}
  ICX - CCX \equiv Y - Z
\end{equation*}
where $Y$ and $Z$ are respectively the truncation of $ICX$ and $CCX$
along their apex edges.  However, $Y$ and $Z$ are respectively the
Cartesian product of a square and a triangle with $X$.

Thus, as an operator on flag vectors, $IC - CC$ is a difference of
Cartesian products.  But $I$ is also a Cartesian product operator, and
such products commute.  The $IC$ equation follows.  In fact, the $IC$
equation together with $I(\point) = C(\point)$ (which follows from the
$IC$ equation applied to the empty set) allows any word in $I$ and $C$
applied to a point to be reduced to an equivalent combination of $C$
and $IC$ words applied to a point.
\end{proof}

This is one motivation for $D = IC - CC$.  The proof of the $IC$
equation establishes:

\begin{corollary}
For any $X$
\begin{equation*}
  DX \> \equiv \> (IC - CC) X \> \equiv \> (Y - Z) \times X
\end{equation*}
where $Y$ and $Z$ are respectively a square and a triangle.
\end{corollary}

\begin{corollary}[The $(C, D)$ basis]
Every convex polytope $X$ has a unique representation
\begin{equation*}
  X \equiv \sum\nolimits_w \lambda_w w(\point)
\end{equation*}
where the sum is over all degree~$d$ words $w$ in $C$ and $D$ and the
coefficients $\lambda_w = \lambda_w(X)$ are linear functions of the
flag vector $f(X)$.
\end{corollary}

\begin{definition}[$CD$-vector]
The mapping $w \mapsto \lambda_w$ is called the \emph{$CD$ vector} of
$X$.  If $s$ is a set of $(C, D)$ words we will use $\lambda_s$ to
denote $\sum_{w\in s}\lambda_w$.
\end{definition}

The zero-one hypothesis implies that each $g_i$ is equal to
$\lambda_s$ for some set of words~$s$.

\section{Homology}

In this section we look at the already known homology.  In particular
we state formulas for Betti numbers in terms of the $(C, D)$ basis.
First simple polytopes.  The product of two simple polytopes is also
simple.  Thus if $X$ is simple then so is $DX$ (by which we mean that
its flag vector can be written as a formal sum of the flag vectors of
simple polytopes).  This follows because $DX$ is equivalent to $(Y -
Z)\times X$ for simple $Y$ and $Z$.

Clearly, $h(\point) = [1]$, $h(C(\point)) = [1, 1]$ and similarly
$h(C^j(\point)) = [1, \ldots, 1]$.  Similarly $h(Y) = [1,1,1]$, $h(Z)
= [1,2,1]$, and $h(DX) = [0, 1, 0]h(X)$.  Thus $h(D^iC^j(\point)) =
[0,1,0]^i[1, \ldots, 1]$.  Also, any simple-polytope $h$-vector is a
weighted sum of $D^iC^j(\point)$ $h$-vectors and so the formal sums
$D^iC^j(\point)$ provide a basis for the flag vectors of simple
polytopes.  This proves:

\begin{proposition}
Suppose $X$ is a simple polytope.  Then
$ 
  X \equiv \sum\nolimits _{i=0}^{\lfloor d/2 \rfloor} \> \lambda_i D^iC^{d-2i}(\point)
$ 
where $\lambda_i = g_i(X)$.
\end{proposition}

For general polytopes it follows from the known formula for
intersection homology Betti numbers~\cite{sta87:gen} that $h(IX) = [1,
  1] h(X)$.  It also follows that for example if $h(X) = [a, b, c, b,
  a]$ then $h(CX) = [a, b, c, c, b, a]$.  The general rule for $C$ is
to repeat the middle term (for $d$ even) or the next to middle term
(for $d$ odd).  If $d$ is odd then the two next-to-middle terms are
equal.  Thus, if $h(X) = [a, b, b, a]$ then $h(CX) = [a, b, b, b, a]$.

From this $h(DX) = [0, 1, 0] h(X)$ follows easily.  (Alternatively,
the result follows from $D=Y-Z$ and the K\"unneth formula.)  We now
turn to the $g$-vector.  The rule for $h(CX)$ then becomes $g(CX) =
g(X)$ while the rule for $D$ is $g(DX) = [0, 1, 0]g(X)$.  Hence we
have the following extension of the previous proposition:

\begin{proposition}
Let $w$ be a word in $C$ and $D$.  Then
\begin{equation*}
  g_i(w(pt)) = (\deg_D w = i)
\end{equation*}
and by generalised Dehn-Sommerville this equation determines $g_i$ on
all convex polytope flag vectors.
\end{proposition}

\begin{corollary}
The known components $g_i$ of $g$ satisfy the zero-one on $(C,D)$
hypothesis.
\end{corollary}

The proof that $g_i(X) \geq 0$ is deep, and relies on the
decomposition theorem~\cite{dcml09:decomp}.

\section{Higher-order homology}
\label{sect:higher}

This section assumes some knowledge of intersection homology and the
topology of toric varieties.  It provides both justification for the
axioms and an interpretation of the outcome of calculations that
follow.

For simple polytopes the homology $H(X)$ is a ring generated by the
facets and from this the pseudo-power inequalities follow.  For
general polytopes $H(X)$ no longer has a ring structure.  Our distant
goal requires additional components in $H(X)$, which we will call the
\emph{higher-order homology}.  The already known part we call the
\emph{order-zero homology}.

The main properties of higher-order homology are: (1)~it should vanish
when $X$ is simple, (2)~when it vanishes there is a ring structure on
$H(X)$, and (3)~its Betti numbers should be a linear function
of the flag vector.

The order-zero homology can be constructed by taking cycles and
relations whose intersection with the strata satisfy what are called
perversity conditions.  This allows Poincar\'e duality and hard
Lefschetz to hold (provided we use middle perversity).  It also allows
non-trivial local cycles to exist.  For example there is a local cycle
at each vertex of the octohedron (the polar of a $3$-cube).

The author expects the cycles and relations for higher-order homology
to be subsets of those for order-zero homology, obtained by imposing
locality conditions.  In other words, fewer generators and fewer
relations.  For example, on the octohedron there are $8$ local
$1$-cycles which under local equivalence are independent.  However,
there are $4$ independent global relations among these local
$1$-cycles.

If $X=v(\point)$ for $v$ a word in $(C,I)$ then the stratification is
given by a single flag on $X$.  This has consequences.  Suppose $\eta$
and $\psi$ are cycles close to a stratum $X_i$.  If $\eta$ and $\psi$
are globally equivalent then it seems that this relation can be
deformed so that it too is close to $X_i$.  In other words, on $X=
v(\point)$ local and global equivalence are the same.  This, of
course, is not true in general.

Further, if $IC = D + CC$ is used to rewrite $v$ as a weighted sum of
words in $(C, D)$ then the weights count the primitive cycles on $X$
with certain locality properties.  Thus, each locality property gives
rise to a zero-one function on the $(C,D)$ basis.  (We will see,
however, that for some plausible locality properties the resulting
linear function will be negative on some polytope.)  To summarise what
this paper needs from this section:

\begin{hypothesis}
The higher-order homology vanishes on simple polytopes.
\end{hypothesis}

\begin{hypothesis}
Each higher-order homology group is defined by using a locality
property.
\end{hypothesis}

\begin{hypothesis}
Each locality property induces a zero-one function on the $(C,D)$
basis.
\end{hypothesis}


\section{Axioms for the $g$-vector}

The previously stated hypotheses have consequences for the $g$-vector,
which we will call axioms.  The calculation is based on the axioms.
Recall that $X$ is any convex polytope, that $d$ is its dimension,
that $w$ is a word of degree~$d$ in $(C, D)$, and that $s$ is a set of
words $w$.  Here are the axioms.

\begin{axiom}[Linearity]
$g(X)$ is a linear function of the flag vector $f(X)$.
\end{axiom}

\begin{axiom}[Components]
$g(X)$ is a map $i \mapsto g_i(X)$ from $I$ to $\mathbb{Z}$, for $i$
  in some index set~$I = I_d$.
\end{axiom}

Each $g_i$ is called a \emph{component} of $g$.

\begin{axiom}[Non-negative]
\label{axiom:nonnegative}
$g_i(X)\geq 0$ for all $i\in I$ and all $X$.
\end{axiom}

The next axiom is crucial.  It implies that $g$ is given by a choice
from a known and finite set.  Its motivation comes from the three
previous sections.  It depends upon the generalised Dehn-Sommerville
equations.

\begin{axiom}[Zero-one on $(C, D)$ words]
For each $i$ in $I$ we have $g_i = \lambda_s$ for some set of words
$s$.
\end{axiom}

Thus, we can take the index set $I$ to be a subset of the power set of
all degree~$d$ words in $C$ and $D$. For example, for $d=5$ we have
$F_{5+1} = 8$ and so the power set has $2^8 = 256$ elements.  Thus,
\emph{a priori} there are $2^{256}$ possible values for $I$.  From now
on we will write $s \in S$ instead of $i \in I$.  The next tasks are
to remove first the impossible and then the redundant.
Axiom~\ref{axiom:nonnegative} provides some conditions.

\begin{definition}
Say that $s$ is \emph{effective} if $\lambda_s(X) \geq 0$ for all $X$.
\end{definition}

Let $E = E_d$ denote the set of all effective $s$ (of degree $d$).
Now for the redundant.  Clearly, if nonempty disjoint word-sets $s$
and $s'$ are both effective then so is their union.  Such unions
provide no new conditions on $f(X)$ and so are excluded from $g$.
More generally:

\begin{definition}[Extremal]
Say that $s$ in $E_d$ is \emph{extremal} if it cannot be expressed as
a weighted sum $s = \sum_{t\neq s} \alpha_t t$ of other elements of
$E_d$ with non-negative weights $\alpha_t$. (The empty set is not
extremal.)
\end{definition}

\begin{axiom}[Extremal]
Each element of $s$ of $S$ are extremal in $E$.
\end{axiom}

The $g$-vector should embrace both zero-order and all possible
higher-order homology.

\begin{axiom}[Order-zero homology]
The $g_i$ corresponding to order-zero homology are components of $g$.
\end{axiom}

\begin{axiom}[Higher-order homology]
If $s$ is effective, extremal, and $\lambda_s$ vanishes on simple
polytopes, then $s \in S$.
\end{axiom}

Finally, we want $g$ to be complete and without redundancy.

\begin{axiom}[Basis]
The components of $g$ are a basis for all linear functions of the flag
vector $f$.
\end{axiom}


\section{Conjectures}

In this section we state two conjectures.  If both are true then for
each dimension $d$ a finite calculation will give a formula for the
$g$-vector in that dimension, and hence linear inequalities
$g_s(X)\geq 0$. The first conjecture is, of course:

\begin{conjecture}
There is a $g$-vector that satisfies the axioms stated in the
previous section.
\end{conjecture}

The set $E=E_d$ of effective word-sets $s$ of degree~$d$ is central to
this conjecture.  The $g$-vector follows once we have $E$.  To show
that $s$ is \emph{not} in $E$ it is enough to produce an $X$ such that
$g_s(X)<0$.  Therefore, a finite set $P$ of test polytopes is by
elimination enough to determine $E$.  We will make this more formal.

\begin{definition}[Effective word-sets]
Let $P$ be a set of dimension~$d$ polytopes.  Define $E(P)$, the
\emph{word-sets effective on $P$}, to be the set of $s$ such that
$\lambda_s(X)\geq 0$ for every $X$ in $P$.
\end{definition}

\begin{definition}[Broad set of polytopes]
Say that a set $P$ of polytopes is \emph{broad} if $E(P) = E$, or in
other words that $\lambda_s(X) \geq 0$ for $X\in P$ implies
$\lambda_s(X)\geq 0$ for all $X$.
\end{definition}

\begin{proposition}
For each dimension $d$ there is a broad and finite set $P$ of
$d$-polytopes.
\end{proposition}

In dimension~$5$ there is a broad $P$ with at most $2^8 = 256$
polytopes.  However, we don't know what it is.  The second conjecture,
if true, provides an explicit broad set $P$.  There is at present
little evidence for it, other than the satisfactory output produced by
the $d=5$ calculation, which is the subject of the next two sections.
The set $P$ is produced from subsets of the vertices of a $d$-cube,
together with polars.  A cube can be represented so that each vertex
component is either $0$ or $1$.  We repeat the standard definition:

\begin{definition}
A \emph{$01$-polytope} is the convex hull of a subset of the vertices
of a cube.
\end{definition}

A cube has $2^d$ vertices and $2d$ facets.  The polar of a cube is the
\emph{cross-polytope}, the convex hull of the $d$ basis vectors $e_i$
and their negatives $-e_i$.  The cross-polytope has $2^d$ facets
(namely a choice of which $e_i$ are to have a negative sign).  Every
convex polytope has a \emph{polar $X^\vee$} (which depends on the
choice of a point $p$ in the interior of $X$).  Each $i$-face on $X$
corresponds to a $d-i-1$ face of $X^\vee$ and vice versa.  This
bijection reverses inclusions and so the flag vector $f(X^\vee)$ of
the polar is a linear function of $f(X)$ (and does not depend on the
choice of $p$).

\begin{definition}
Let $P_{01,d}$ denote the set of all $d$-dimensional $01$-polytopes,
together with their polars.
\end{definition}

A $01$-polytope need not have the same dimension as its cube.  But if
of smaller dimension then there is a projection onto a face that gives
an affinely equivalent polytope (which thus has the same flag vector).
So without loss of generality we need only consider $X$ and $X^\vee$
obtained from the $d$-cube.

\begin{conjecture}[$P_{01,d}$ is broad]
The set $P_{01,d}$ of $d$-dimensional $01$-polytopes and their polars
is broad.
\end{conjecture}


\section{The $d=5$ calculation}

In the previous section we conjectured that certain calculations would
give a formula for the $g$-vector.  In this section and the next we
report on some such calculations.  We applied
\texttt{polymake}~\cite{polymake} to a list of $01$-polytopes for
$d=5$ supplied by Aichholzer~\cite{aich10:personal}.  This gave us a
list of flag vectors.  There are 1,226,525 such polytopes (up to cube
symmetries) and the calculation took about 15 GHz days.  Together with
their polars this gives 688,298 distinct flag vectors.

We also wrote software~\cite{fine10:bbpoly} to compute the $(C,D)$
vector for each of these flag vectors, and also to determine for each
of the $256 = 2^8$ subsets of the $F_{5+1} = 8$ degree~$5$ words $s$
in $(C, D)$ whether or not the sum of the corresponding coefficients
in the $(C,D)$ vector was non-negative for all these flag vectors.
This gives us $E(P_{01,d})$, as defined in the previous section.

Finally, we used \texttt{polymake} again to determine the extremal $s$
in $E(P_{01,d})$.  Altogether there were $13$ such, of which we are
expecting $8$ to be components of the $g$-vector.  The meaning of the
remaining $5$ is less clear.  The next two results give the
conjectured components of the $g$-vector.

\begin{proposition}
For $P_{01,5}$ the order-zero $g_i$ (in the form `index $=$ deg :
word-set')
\begin{align*}
11111 = 0 &: CCCCC\\
2111 = 1 &: DCCC,\  CDCC,\ CCDC,\ CCCD\\
221 = 2 &: DDC,\ DCD,\ CDD
\end{align*}
are effective (this is already known) and extremal.
\end{proposition}

\begin{proposition}
For $P_{01,5}$ the higher-order $g_s$ (in the form `index : word-set')
are
\begin{align*}
1211 &: CDCC,\ CCDC,\ CCCD,\ CDD\\
1121 &: CCDC,\ CCCD,\ DCD\\
1112 &: CCCD\\
122 &: CDD\\
212 &: DCD
\end{align*}
\end{proposition}

A word about the notation above, which is also used later.  To each
$s$ we associate an index, which is a word in the symbols $1$ and $2$.
Thus, $g_{221}= g_2 = g_s$ where $s=\{DDC, DCD, CDD\}$.  The index
$i$, with $C$ and $D$ replacing $1$ and $2$, is the first listed
element of the set $s$. This notation is concise and, with care, the
context resolves any ambiguities.

We will now show that the above are a basis for all linear functions
of the flag vector.  Let $w$ be a word.  It is enough to show that
$\lambda_w$ is a linear combination of the $g_i$ above.  But this is
easy, because the components are listed in an upper-triangular order.
In other words, every word appears first in one of the sets, and it
does not appear in any subsequent set. This proves:

\begin{proposition}
The above components (of the conjectured $g$-vector) are a basis for
all linear functions of the flag vector.
\end{proposition}

We now turn to the remaining effective and extremal $s$.  They do not
vanish on simple polytopes, and they are not order-zero $g_i$
numbers. This is why they are not part of the $g$-vector.

\begin{proposition}
For $P_{01,5}$ the remaining effective and extremal $s$ (in the form
`inequality : word-set') are
\begin{align*}
g_{122} \leq g_{221} + g_{2111} &\quad :\quad DDC,\ DCD,\ DCCC,\ CDCC,\ CCDC,\ CCCD\\
g_{122} \leq g_{221} + g_{1211} &\quad :\quad DDC,\ DCD,\ CDD,\ CDCC,\ CCDC,\ CCCD\\
g_{212} \leq g_{221} + g_{1112} &\quad :\quad DDC,\ CDD,\ CCCD\\
g_{212} \leq g_{221} + g_{1121} &\quad :\quad DDC,\ DCD,\ CDD,\ CCDC,\ CCCD\\
g_{212} \leq g_{221} + g_{2111} &\quad :\quad DDC,\ CDD,\ DCCC,\ CDCC,\ CCDC,\ CCCD
\end{align*}
where the left column is $\lambda_s \geq 0$ written in terms of the components
of the $g$-vector.
\end{proposition}

These inequalities, an unanticipated by-product of the calculation,
are more concisely be written as
\begin{align*}
g_{122} - g_{221} &\leq \min (g_{2111}, g_{1211})\\
g_{212} - g_{221} &\leq \min (g_{2111}, g_{1121}, g_{1112})
\end{align*}
and in this form they are similiar to McMullen's pseudo-power
inequalities $g_{i+1}\leq g_i^{\langle i \rangle}$.

\section{A special test polytope}

The $d=5$ calculation gives not only a formula for $g$ but also,
implicitly, a small and broad set of test polytopes.  Here is the most
interesting example. The set
\begin{equation*}
 s = \{ CDCC,\ CCDC,\ CCCD \}
\end{equation*}
has $\lambda_s = g_{1211} - g_{122}$ and is effective on the whole of
$P_{01,5}$ with just one exception (up to symmetries of the cube).

In this section let $X$ be the $01$-polytope which has all the
vertices of the $5$-cube except the set $V$ whose members, listed in
lexicographic order, are
\begin{align*}
&
u_1 = (0,0,0,0,0),\quad
u_2 = (0,0,0,1,1),\quad
u_3 = (0,1,1,0,0),\quad
u_4 = (0,1,1,1,1),
\\\relax
&
v_1 = (1,0,1,0,1),\quad
v_2 = (1,0,1,1,0),\quad
v_3 = (1,1,0,0,1),\quad
v_4 = (1,1,0,1,0)\>.
\end{align*}

The exception is the polar $X^\vee$ of $X$. The polytope $X^\vee$ has
flag vector (independent components in \texttt{polymake} order)
\begin{equation*}
f (X ) = (1, 24, 112, 152, 464, 80, 400, 696)
\end{equation*}
while the  $(C, D)$ vector is given by
\begin{align*}
&(CCCC, 1),\quad (CCCD, 8),\quad (CCDC, 56),\quad (CDCC, -66),
\\
&(CDD, 20),\quad (DCCC, 20)\quad (DCD, 0),\quad (DDC, -5)
\end{align*}
and so
\begin{equation*}
\lambda_s(X^\vee) = \lambda_{CDCC} + \lambda_{CCDC} + \lambda_{CCCD} =
-66 + 56 + 8 = -2
\end{equation*}
which is, as claimed, negative.  (This calculation, repeated many
times, gave us $E(P_{01,5})$.

The $g$-vector of $X^\vee$, however, is
\begin{align*}
&g_0 = 1,
\quad g_1 = 8 + 56 - 66 + 20 = 18,
\quad g_2 = 20 + 0 - 5 = 15,\\
&g_{1211} = 8 + 56 - 66 + 20 = 18,\quad
g_{1121} = 8 + 56 + 0 = 64,\quad
g_{1112} = 8,\quad
g_{122} = 20, \quad
g_{212} = 0
\end{align*}
whose components are (for this polytope by construction of $g$)
non-negative.

The set $V$ of missing vertices has a structure.  The $u_i$ lie on a
$4$-face, and the $v_i$ on the opposite $4$-face.  The vertices $u_1$
and $u_4$ give a diagonal of the $4$-face, as do $u_2$ and $u_3$.  The
vertex pairs $\{u_1, u_2\}$, $\{u_1, u_3\}$, $\{u_2, u_4\}$ and
$\{u_3, u_4\}$ all differ in two coordinates.  The same is true for
the $v_i$.  Finally, each $u_i$ differs from each $v_j$ in three
coordinates. To state this more formally, let $d_1$ denote the Hamming
(or Manhattan) metric on the vertices of the cube; $d_1$ counts how
many components differ. The following definition and proposition
summarise these facts.

\begin{definition}[Distance count]
Let $V'$ be a subset of the vertices of a cube.  For each vertex $v$
of the cube the $i$-th component of the \emph{distance count of $v$
  from $V'$} is the number of $v'\in V'$ with $d_1(v, v') = i$.
\end{definition}

\begin{proposition}
Let $v$ be any member of the missing vertices set $V$.  Then the
distance count of $v$ from $V$ is $(1,0,2,4,1,0)$.
\end{proposition}

This result has a converse.  It tells us that, at least in this case,
distance count can determine a vertex set.  As construction of
polytopes with special properties may be useful later, we state the
converse and give the somewhat pedestrian proof.

\begin{proposition}
Up to symmetry there is exactly one subset $V$ of the $5$-cube that
has distance count $(1, 0, 2, 4, 1, 0)$.
\end{proposition}

\begin{proof}
Let $V$ be the subset above and let $V'$ be another one.  Pick a
vertex in $V'$.  By cube symmetry without loss of generality (wlog) it
is $u_1$.  There's only one vertex at distance $4$ from $u_1$ and wlog
is is $u_4$.  Pick a point at distance $2$ from $u_1$.  Its distance
from $u_4$ \emph{a priori} is $2$ or $4$. But $d_1(u_1, u_4)=4$ and so
by distance count it must be $2$ and wlog the point is $u_2$.

Now pick the other point at distance $2$ from $u_1$.  As before, its
first coordinate is zero.  If not $u_3$ then it has distance $2$ from
$u_1$, $u_2$ and $u_4$.  This violates the distance count.  So it is
$u_3$.

Now choose a vertex $v$ at distance $3$ from all the $u_i$.  We have
$d_1(u_1, v) = 3$ and so it contains $3$ ones. If the first component
is $0$ then $d_1(v_4, v) = 1$. So the first component is~$1$.  It
cannot be $(1, 0, 0, 1, 1)$ or $(1, 1, 1, 0, 0)$ because of $u_2$ and
$u_5$ respectively. Therefore the $v_i$ are the only vertices at
distance~$3$ from all the $u_i$.
\end{proof}

Finally, we restate these results putting the focus on $V$.

\begin{proposition}
The distance count $(1, 0, 2, 4, 1, 0)$ determines, by complement and
polarisation, the unique up to symmetry polytope in $P_{01,5}$ that
excludes $\{CDCC, CCDC, CCCD\}$.
\end{proposition}


\section{Summary}

The section reviews the rest of the paper and sketches possibilities
for future work.

The key new concepts are (1)~that the components of the $g$-vector are
zero-one on the $(C, D)$ basis, and (2)~the set $E = E_d$ of effective
word-sets given by $\{s | \lambda_s(X)\geq 0$ for all $X \}$.  They
arose as follows. The author computed $\lambda_s(X)$ for all $X$ in
$P_{01,5}$ and for the $s$ in~\cite{fine10:complete}.  This was to
test the $s$ for being effective.  However, as we have seen, some came
out negative.  The author then inverted the process.  He used
$P_{01,5}$ to compute $E(P_{01,5})$ and hence obtained a conjecture.
The structure of $E$, even if not as described by the axioms, is most
likely worth studying.  The conjectures make definite statements about
$E(P_{01,d})$ that can be tested by finite calculation.

Two of the components of $g$ for $d=5$, namely
\begin{align*}
1211 &: CDCC,\ CCDC,\ CCCD,\ CDD\\
1121 &: CCDC,\ CCCD,\ DCD
\end{align*}
have mixed degree in $D$.  The meaning of the $\deg_D=2$ terms, $CDD$
and $DCD$ respectively, is unclear.  Although they may look
improbable, they are an unavoidable consequence of our axioms and
conjectures.  The author expects that in general the lowest degree
terms will be given as in~\cite{fine10:complete}.

The meaning of the conjectured inequalities
\begin{align*}
g_{122} - g_{221} &\leq \min (g_{2111}, g_{1211})\\
g_{212} - g_{221} &\leq \min (g_{2111}, g_{1121}, g_{1112})
\end{align*}
is far from clear.  The analogous growth conditions in the simple case
follow from $H(X)$ being generated by the facets of $X$.  It is
possible that \cite{fine98:uniform} will help here.

Further calculations would be helpful.  However, the size of
$P_{01,d}$ grows very rapidly with $d$.  It might be practical to
compute for the whole of $P_{01,6}$ but for $P_{01,7}$ only a subset
can be used.  Analysis of the $d=4$ and $5$ results may lead to
smaller broad subsets of $P_{01,d}$.  This is of course related to the
construction of polytopes and thus the proof of sufficiency.  Results
in dimension~$4$ should be compared to known results, conjectures and
constructions~\cite{bayerlee94:aspects,bayer87:extended,pafwern05:construction}.

Simplicial poltyopes are a useful special case.  Their polars are
simple, and so McMullen's conditions characterise their flag vectors.
It may be instructive, say for $d=5$, to write the already known
conditions on simplicial polytopes in terms of the conjectured
$g$-vector.

The extension of McMullen's conditions to general polytopes will have
at least three components, namely: (1)~linear equations on the flag
vector, (2)~linear equalities $g_s(X)\geq 0$, and (3)~growth
conditions on the $g_s$.  The generalised Dehn-Sommerville equations
of Bayer and Billera provide~(1).  Our first conjecture, if true,
provides~(2) abstractly while our second conjecture will then provide
a finite calculation that gives~(2) in a concrete form.  It may also
provide part of~(3).

Here are three more distant goals: (1)~produce, with supporting
evidence, a conjectured formula for $g$ in all dimensions,
(2)~systematically produce test polytopes, particularly in higher
dimensions, (3)~translate the formulas and conjectures in this paper
to intersection homology.  This last goal is probably the only way to
prove the conjectures, probably via the decomposition
theorem~\cite{dcml09:decomp}.

\bibliographystyle{amsplain}
\bibliography{axioms-for-g-vector.bib,../../2009/ehv/bayer-master.bib}

\rightline{Email: \texttt{jfine@pytex.org}}

\end{document}